\documentclass[final]{siamart190516}

\usepackage{amsfonts}
\usepackage{graphicx}
\usepackage{epstopdf}
\usepackage{algorithmic}
\ifpdf
  \DeclareGraphicsExtensions{.eps,.pdf,.png,.jpg}
\else
  \DeclareGraphicsExtensions{.eps}
\fi

\newsiamremark{remark}{Remark}
\newsiamremark{hypothesis}{Hypothesis}
\crefname{hypothesis}{Hypothesis}{Hypotheses}
\newsiamthm{claim}{Claim}
\newsiamthm{assumption}{Assumption}
\newsiamremark{example}{Example}

\headers{Average cost per unit time impulse control}{{\L}. Stettner}

\title{ On an approximation of average cost per unit time impulse control of Markov processes}

\author{{\L}ukasz Stettner\thanks{Institute of Mathematics, Polish Academy of Sciences, Warsaw, Poland,
  (\email{l.stettner@impan.pl});  research supported by NCN grant 2016/23/B/ST1/00479.}}

\usepackage{amsopn}

\usepackage{enumitem} 

\newcommand{\vep}{\varepsilon}
\DeclareMathOperator*{\argmin}{arg\,min}

\def\ee{\mathbb{E}}

\def\prob{\mathbb{P}}
\def\tao{\tau_{\cal{O}}}
\def\om{{\cal{O}}_m}
\def\taom{\tau_{{\cal{O}}_m}}
\def\taomk{\tau_{{\cal{O}}_{m_k}}}
\def\ep{\epsilon}
\newcommand\ind[1]{1_{#1}}

\makeatletter
\def\namedlabel#1#2{\begingroup
    #2%
    \def\@currentlabel{#2}%
    \phantomsection\label{#1}\endgroup
}
\makeatother

\begin{document}

\maketitle

\begin{abstract}
In this paper we consider impulse control of continuous time Markov processes with average cost per unit time functional.
This problem is approximated using impulse control problems stopped at the first exit time from increasing sequence of open sets.  We find solution to Bellman equation corresponding to the original problem and show that stopped impulse control problems approximate optimal value of the cost functional.
\end{abstract}

\begin{keywords}
 Optimal stopping, Bellman equation, impulse control, average cost per unit time
\end{keywords}

\begin{AMS}
 93E20, 49N60, 93C10, 60J25
\end{AMS}

\section{Introduction}\label{S:introduction}
Let $(X_t)$ be a Feller-Markov process on $(\Omega, F,(F_t))$ taking values in a locally compact space $E$ with metric $\rho$ and Borel $\sigma$ field ${\cal E}$. The process starting from $x$ at time $0$ generates a probability measure $\prob_x$ and we denote by $\ee_x$ a related expectation operator. We shall assume that Markov semigroup $T_t$ defined for $h\in C(E)$, the space of continuous bounded functions on $E$, as $T_th(x)=\ee_x\left\{h(X_t)\right\}$ transforms continuous functions vanishing at infinity $C_0(E)$ into itself.  Process $(X_t)$ is controlled by impulses $V=(\tau_i, \xi_i)$ consisting of an increasing sequence of stopping times $\tau_i$ and impulses $\xi_i$, which is adapted to the available observation till time $\tau_i$: at time $\tau_i$ the process is shifted from the state $X_{\tau_i}$ to the state $\xi_i$ at the cost of $c(X_{\tau_i}, \xi_i)\geq c>0$ and follows its dynamics until the next impulse. We assume that impulses shift the process to a compact set $U \subseteq E$, i.e., $\xi \in U$ and the cost function $c(x,\xi)$ is positive, continuous bounded and uniformly bounded away from zero by a constant $c>0$. Furthermore to avoid multiple impulses we assume that for $x\in E$ and $\xi,\xi'\in U$ we have $c(x,\xi)\leq c(x,\xi')+c(\xi',\xi)$. A strategy $V=(\tau_i, \xi_i)$ is \emph{admissible} for $x \in E$ if $\tau_i$ form an increasing sequence of stopping times (possibly taking the value $\infty$) with $\lim_{i \to \infty} \tau_i = \infty$, $\prob_x$-a.s. To describe the evolution of the controlled process we introduce a construction of \cite[Section 2]{Stettner1983a}, which follows ideas of \cite{Rob1978}. Namely, we consider $\Omega=D(R^+,E)^\infty$, where $D(R^+,E)$ is a canonical space of right continuous, left limited functions on $R^+$ taking values in $E$. We assume that $(F^1_t)$ is a canonical filtration on $D(R^+,E)$  and inductively define $F^{n+1}_t=F^n_t \otimes F_t$. The stopping times  $\tau_i$ are adapted $F^i_t \times \left\{\emptyset,D(R^+,E)\right\}^\infty$ while the impulses $\xi_i$ are adapted to $F^i_{\tau_i} \times \left\{\emptyset,D(R^+,E)\right\}^\infty$.
The trajectory of the controlled process $(Y_t)$ is defined using coordinates $x^n$ of the canonical space $\Omega$, i.e. $Y_t=x^n_t$ for $t\in [\tau_{n-1},\tau_n)$, with $\tau_0=0$. Given an impulse strategy $V$ following \cite[Section 2]{Stettner1983a} and \cite[Chapter 5 and Appendix 2]{Rob1978} we define a probability measure $\prob^V$ on $\Omega$.
Our goal is to minimize over all admissible strategies the functional
\begin{equation}\label{eqn:weaker_functional}
{J}\big(x, (\tau_i, \xi_i)\big) = \liminf_{n \to \infty} \frac{1}{\ee_x^V\{\tau_n\}} \ee_x^V \Big\{ \int_0^{\tau_n} f(Y_s) ds + \sum_{i = 1}^n c(x_{\tau_i}^i, \xi_i) \Big\},
\end{equation}
assuming that  $(\tau_n)$ are such that $\ee_x\{\tau_n\}<\infty$,
where $f$ is a continuous bounded function and $x_{\tau_i}^i$ is the state of the process before the $i$-th impulse with a natural meaning if there is more than one impulse at the same time. Alternatively we can consider average cost per unit time functional
\begin{equation}\label{eqn:functional}
\hat{J}\big(x, (\tau_i, \xi_i)\big) = \liminf_{T \to \infty} \frac{1}{T} \ee_x^V \Big\{ \int_0^{T} f(Y_s) ds + \sum_{i = 1}^\infty \ind{\tau_i\leq T} c(x_{\tau_i}^i, \xi_i) \Big\}.
\end{equation}

Impulse control of random systems is intensively studied from various points of view. Such control problems have many applications: in finance (cash management, portfolio optimization) \cite{Korn}, \cite{Belak}, control of the exchange rate \cite{Mun}, optimal harvesting \cite{Wil}, \cite{Alv}, inventory control \cite{Ben2012}  (see also the first chapter of \cite{BenLio1984} for other applications). Impulse control of diffusions and diffusions with jumps was studied in \cite{BenLio1984}, \cite{Bay2013}. Average cost per unit time impulse control problems were studied first for fixed cost of impulses in \cite{Rob1981}, \cite{Robin1983} under uniform ergodicity assumption and then for so called separated cost in \cite{Stettner1983a} and in \cite{Ste1986} under quasicompactness of transition semigroup. The problem was also studied under some compactness conditions in \cite{GS}. Ergodic impulse control of diffusion processes in a bounded domain was studied in \cite{LionsPert1986} and \cite{Pert1988}. An extension to an unbounded domain (real line) was then considered in \cite{JZ}. Average cost per unit time impulse control for Levy process was studied in \cite{Ch202}. In the case of compact state space and constraints the average cost per unit time impulse control problem was studied in \cite{MenRob2018}. This kind of cost functional is also important in mathematics of finance under the name of Kelly criterion - see \cite{Mac2011}. In particular impulse control appear in the case of growth optimal portfolio under proportional transaction costs in \cite{DPS} and \cite{Ch2017} and references there in. Average cost per unit time impulse control in a general (locally compact separable) state space was considered in \cite{PalSte2017}. This paper is an extension of the last paper. Namely in the paper \cite{PalSte2017} discounted approximation approach was considered. In this paper we consider another approach based on stopped impulse control with increasing sequence of domains. We use stopped impulse control approach i.e. we stop the controlled process when we enter a boundary of an open set. We consider a sequence of open sets increasing to the whole space and obtain a sequence  of impulse control problems which as we show approximate original nonrestricted control problem. This way we obtain a nice approximation procedure to original impulse control problem. As one can see e.g. in \cite{Ben2012} unrestricted impulse control problem with average cost per unit time functional is difficult to solve numerically even for rather simple inventory models. The method proposed in the paper seems to be natural approximation procedure which justifies computational intuition. In the paper we use stopping results from the paper \cite{PalSte2011}.
We consider first the functional in a weak form \eqref{eqn:weaker_functional} and then study its general form \eqref{eqn:functional}.
Several nontrivial results concerning ergodic optimal stopping are of independent interest and are shown in the appendix.

\section{Stopped impulse control problems}\label{sec stopped problem}
Let $\cal{O}$ be an open set in $E$ such that $\sup_{x\in \cal{O}}\ee_x\left\{(\tau_{\cal{O}})^2\right\}<\infty$, for $\tau_{\cal{O}}=\inf\left\{s\geq 0: X_s\in E\setminus \cal{O}\right\}$, which means that $\tau_{\cal{O}}$ is the first exit time from $\cal{O}$. We shall assume that stopped semigroup $T_t^{\cal{O}}h(x):=\ee_x\left\{h(X_t)\ind{t<\tau_{\cal{O}}}\right\}$ transforms $C(E)$ into itself.
Consider stopped impulse functional

\begin{equation}\label{stopped1}
{J}^{\cal{O}}\big(x, (\tau_i, \xi_i)\big)= \liminf_{n \to \infty} \frac{1}{\ee_x^V\left\{\tau_n\wedge \tao \right\}} \ee_x^V \Big\{ \int_0^{\tau_n\wedge \tao} f(Y_s) ds + \sum_{i = 1}^n \ind{\tau_i < \tau_{\cal{O}}}c(x_{\tau_i}^i, \xi_i) \Big\}.
\end{equation}


Let for $\alpha>0$

\begin{equation}\label{eqn.1}
\lambda_\alpha:=\inf_{x\in U}\inf_{V} {\ee_x^{V}\Big\{\int_0^{\tao}e^{-\alpha s} f(Y_s)ds + \sum_{i = 1}^\infty \ind{\tau_i < \tau_{\cal{O}}}e^{-\alpha \tau_i} c(x_{\tau_i}^i, \xi_i)\Big\} \over \ee_x^{V}\Big\{\int_0^{\tao}e^{-\alpha s} ds\Big\}},
\end{equation}
\begin{equation}\label{eqn.2}
w_\alpha(x):= \inf_{V} \ee_x^{V}\Big\{\int_0^{\tao} e^{-\alpha s} (f(Y_s)-\lambda_\alpha) ds + \sum_{i = 1}^\infty \ind{\tau_i < \tau_{\cal{O}}} e^{-\alpha \tau_i} c(x_{\tau_i}^i, \xi_i)\Big\},
\end{equation}
where $V$ stands for impulse strategy $(\tau_i,\xi_i)$.
We clearly have that

\begin{lemma}\label{lem1} For each $\alpha>0$ we have
\begin{equation}\label{eq2}
\inf_{x\in U} w_\alpha(x)=0.
\end{equation}
\end{lemma}
\begin{proof}
For  $x\in U$ and any impulse strategy $V$ from \eqref{eqn.1} we have that
\begin{equation}
0\leq \ee_x^{V}\Big\{\int_0^{\tao}e^{-\alpha s} (f(Y_s)-\lambda_\alpha) ds + \sum_{i = 1}^\infty \ind{\tau_i < \tau_{\cal{O}}} e^{-\alpha \tau_i} c(x_{\tau_i}^i, \xi_i)\Big\}.
\end{equation}
Consider now $x_{\ep}\in U$ and $V^{\ep}$ such that
\begin{equation}
\lambda_\alpha\geq -\ep+ {\ee_{x_{\ep}}^{V^{\ep}}\Big\{\int_0^{\tao} e^{-\alpha s}f(Y_s)ds + \sum_{i = 1}^\infty \ind{\tau_i < \tau_{\cal{O}}} e^{-\alpha \tau_i}c(x_{\tau_i}^i, \xi_i)\Big\} \over \ee_{x_{\ep}}^{V^{\ep}}\Big\{\int_0^{\tao}e^{-\alpha s} ds\Big\}}.
\end{equation}
Then
\begin{equation}
0\leq \ee_{x_{\ep}}^{V^{\ep}}\Big\{\int_0^{\tao} e^{-\alpha_s} (f(Y_s)-\lambda_\alpha)ds + \sum_{i = 1}^\infty \ind{\tau_i < \tau_{\cal{O}}} e^{-\alpha \tau_i} c(x_{\tau_i}^i, \xi_i)\Big\} \leq \ep  \ee_{x_{\ep}}^{V^{\ep}}\left\{\int_0^{\tao}e^{-\alpha s} ds\right\}
\end{equation}
and since $\ee_{x_{\ep}}^{V^{\ep}}\left\{\int_0^{\tao}e^{-\alpha s} ds\right\}\leq {1 \over \alpha}$
we complete the proof.
\end{proof}

We characterize now the discounted value function $w_\alpha$.
\begin{proposition}
Function $w_\alpha$ defined in \eqref{eqn.2} is a solution to the following Bellman equation
\begin{equation}\label{eq3}
w_\alpha(x)=\inf_\tau \ee_x\left\{\int_0^{\tau\wedge \tao} e^{-\alpha s} (f(X_s)-\lambda_\alpha)ds +  \ind{\tau< \tau_{\cal{O}}}e^{-\alpha \tau} Mw_\alpha(X_\tau)\right\},
\end{equation}
where $Mv(x)=\inf_{\xi\in U}\left[c(x,\xi)+v(\xi)\right]$ for Borel measurable function $v$. Furthermore $w_\alpha \in C(E)$,
\begin{equation}\label{eq4}
|w_\alpha(x)|\leq \|f-\lambda_\alpha\| \ee_x\left\{\tao\right\}
\end{equation}
 and $|\lambda_\alpha|\leq \|f\|$.
\end{proposition}
\begin{proof}
Let for $v\in C(E)$
\begin{equation}
Fv(x):=\inf_\tau \ee_x\left\{\int_0^{\tau\wedge \tao} e^{-\alpha s} (f(X_s)-\lambda_\alpha)ds +  \ind{\tau< \tau_{\cal{O}}}e^{-\alpha \tau} Mv(X_\tau)\right\}.
\end{equation}
When $v(x)\geq 0$ for $x\in U$ then also $Mv(x)\geq 0$ for $x\in E$ and by Theorem \ref{optd} we have that $Fv\in C(E)$. Consider a sequence
$w_\alpha^0(x)=\ee_x\left\{\int_0^{\tao} e^{-\alpha s} (f(X_s)-\lambda_\alpha)ds\right\}$, $w_\alpha^1(x)=Fw_\alpha^0(x)$, $w_\alpha^{n+1}(x)=Fw_\alpha^n(x)$ for $n=1,\ldots$.
One can show that
\begin{equation}
w_\alpha^n(x)=\inf_{V^n}\ee_x^{V^n}\Big\{\int_0^{\tao} e^{-\alpha s} (f(Y_s)-\lambda_\alpha) ds + \sum_{i = 1}^n \ind{\tau_i < \tau_{\cal{O}}} e^{-\alpha \tau_i} c(x_{\tau_i}^i, \xi_i)\Big\},
\end{equation}
where $V^n$ is an impulse strategy consisting of at most $n$ impulses. Clearly $w_\alpha^n(x)$ is a decreasing sequence converging to $w_\alpha(x)$.

Since
\begin{equation}
w_\alpha^{n+1}(x)=\inf_\tau \ee_x\left\{\int_0^{\tau\wedge \tao} e^{-\alpha s} (f(X_s)-\lambda_\alpha)ds +  \ind{\tau< \tau_{\cal{O}}}e^{-\alpha \tau} Mw_\alpha^n(X_\tau)\right\}
\end{equation}
letting $n\to \infty$ we obtain \eqref{eq3}. Notice that $\inf_{x\in U}w_\alpha(x)=0$ and therefore $Mw_\alpha(x)\geq 0$ for $x\in E$ and $Mw_\alpha\in C(E)$ and finally using Theorem \ref{optd} we have that $w_\alpha\in C(E)$.
Furthermore since $c(x,\xi)\geq 0$ we have
\begin{eqnarray}
&&w_\alpha(x)\leq \inf_{V} \ee_x^{V}\Big\{\int_0^{\tao} e^{-\alpha s} \|f-\lambda_\alpha\| ds + \sum_{i = 1}^\infty \ind{\tau_i < \tau_{\cal{O}}} e^{-\alpha \tau_i} c(x_{\tau_i}^i, \xi_i)\Big\}= \nonumber \\
&&\ee_x\Big\{\int_0^{\tao} e^{-\alpha s} \|f-\lambda_\alpha\| ds\Big\}\leq  \|f-\lambda_\alpha\| \ee_x\left\{\tao\right\}
\end{eqnarray}
and similarly
\begin{eqnarray}
&&w_\alpha(x)\geq \inf_{V} \ee_x^{V}\Big\{-\int_0^{\tao} e^{-\alpha s} \|f-\lambda_\alpha\| ds + \sum_{i = 1}^\infty \ind{\tau_i < \tau_{\cal{O}}} e^{-\alpha \tau_i} c(x_{\tau_i}^i, \xi_i)\Big\}= \nonumber \\
&&\ee_x\Big\{-\int_0^{\tao} e^{-\alpha s} \|f-\lambda_\alpha\| ds\Big\}\geq  -\|f-\lambda_\alpha\| \ee_x\left\{\tao\right\},
\end{eqnarray}
where in both approximations the infimum was obtained for a strategy without impulses,
so that we have \eqref{eq4}. The fact that $|\lambda_\alpha|\leq \|f\|$ follows directly from \eqref{eqn.1}.
\end{proof}

We consider now undiscounted impulse control problem.

\begin{theorem}\label{thm1}
We have that
\begin{equation}\label{eq5}
\lim_{\alpha\to 0}\lambda_\alpha=\lambda:=\inf_V \inf_{x\in U}{J}^{\cal{O}}\big(x, (\tau_i, \xi_i)\big)
\end{equation}
and there is a function $w\in C(E)$ which is a solution to the equation
\begin{equation}\label{eq6}
w(x)=\inf_\tau \ee_x\left\{\int_0^{\tau\wedge \tao} (f(X_s)-\lambda)ds +  \ind{\tau< \tau_{\cal{O}}}Mw(X_\tau)\right\}
\end{equation}
such that $\inf_{x\in U}w(x)=0$.
\end{theorem}
\begin{proof}
Since $|\lambda_\alpha|\leq \|f\|$ one can choose a subsequence $\alpha_n\to 0$ and $\lambda$ such that $\lambda_{\alpha_n}\to \lambda$, as $n\to \infty$.
By  \eqref{eq4} we know that $w_\alpha$ are uniformly (in $\alpha$) bounded by a constant $L$ (since $\sup_x \ee_x\left\{\tao\right\}<\infty$). Since $|Mw_\alpha(x)-Mw_\alpha(y)|\leq \sup_{\xi\in U}|c(x,\xi)-c(y,\xi)|$ the family $\left\{Mw_\alpha(x), \alpha>0\right\}$ is equicontinuous at each point and bounded by $\|c\|+L$. Therefore there is a  function $v\in C(E)$ and a further subsequence of $\alpha_n\to 0$ for simplicity again denoted by $\alpha_n$ such that  $Mw_{\alpha_n}$ converges uniformly on compact sets to $v$. Let
\begin{equation}
\hat{w}_\alpha(x)=\inf_\tau \ee_x\left\{\int_0^{\tau\wedge \tao} e^{-\alpha s} (f(X_s)-\lambda_\alpha)ds +  \ind{\tau< \tau_{\cal{O}}}e^{-\alpha \tau} v(X_\tau)\right\}.
\end{equation}
For $x$ from a compact set $K$ and given $\vep>0$ and $T>0$ such that $\sup_{x\in E}{\ee_x\left\{\tao\right\}\over T}\leq \vep$ one can find by Proposition 2.1 of  \cite{PalSte2010} $R>0$ such that
\begin{equation}\label{eq6'}
\sup_{x\in K}\prob_x\left\{\exists_{t\leq T} \ \rho(x,X(t))\geq R\right\}\leq \vep ,
\end{equation}
where $\rho$ is the metric on $E$. Since $B(K,R)=\left\{x: \rho(x,K)\leq R\right\}$ is a compact set, for a sufficiently large $n$, say $n\geq n_0$ we have
$\sup_{y\in B(K,R)} |Mw_{\alpha_m}(y)-v(y)|\leq \vep$. Therefore for $x\in K$
\begin{eqnarray}
&&|w_{\alpha_n}(x)-\hat{w}_{\alpha_n}(x)|\leq \sup_\tau \ee_x\left\{\ind{\tau< \tau_{\cal{O}}}|Mw_{\alpha_n}(X_\tau)-v(X_\tau)|\right\}\leq \nonumber \\
&&  \sup_\tau \ee_x\left\{\ind{\tau< \tau_{\cal{O}}}\ind{\tau_{\cal{O}}\leq T} \ind{X_\tau\in B(K,R)} \vep + (\ind{\tau_{\cal{O}}\geq T}+ \right. \nonumber \\
&& \ind{\tau< \tau_{\cal{O}}}\ind{\tau_{\cal{O}}\leq T}\ind{X_\tau\notin B(K,R)}) \left.(\|c\|+L)\right\} \nonumber \\
&&\leq \vep(1+2(\|c\|+L)).
\end{eqnarray}
Moreover since $|\lambda_\alpha|\leq \|f\|$ and $v$ is bounded we have
\begin{eqnarray}
&&|\inf_\tau \ee_x\left\{\int_0^{\tau\wedge \tao} e^{-\alpha s} (f(X_s)-\lambda_\alpha)ds +  \ind{\tau< \tau_{\cal{O}}}e^{-\alpha \tau} v(X_\tau)\right\} - \nonumber \\
&& \inf_\tau \ee_x\left\{\int_0^{\tau\wedge \tao} (f(X_s)-\lambda_\alpha)ds +  \ind{\tau< \tau_{\cal{O}}} v(X_\tau)\right\}|\leq  \nonumber \\
&& \sup_\tau \ee_x\left\{\int_0^{\tau\wedge \tao} |e^{-\alpha s}-1| |f(X_s)-\lambda_\alpha|ds + \right. \nonumber \\
&& \left. \ind{\tau< \tau_{\cal{O}}}|e^{-\alpha \tau}-1| |v(X_\tau)| \right\}\to 0
\end{eqnarray}
uniformly as $\alpha\to 0$. Since
\begin{eqnarray}
&&|\inf_\tau \ee_x\left\{\int_0^{\tau\wedge \tao} (f(X_s)-\lambda_\alpha)ds +  \ind{\tau< \tau_{\cal{O}}} v(X_\tau)\right\}- \nonumber \\
&& \inf_\tau \ee_x\left\{\int_0^{\tau\wedge \tao} (f(X_s)-\lambda)ds +  \ind{\tau< \tau_{\cal{O}}} v(X_\tau)\right\}| \nonumber \\
&&\leq  |\lambda_\alpha-\lambda|\ee_x\left\{\tao\right\}
\end{eqnarray}
we finally have that $w_{\alpha_n}(x)$ converges uniformly on compact subsets to $w(x)$ of the form
\begin{equation}
w(x)= \inf_\tau \ee_x\left\{\int_0^{\tau\wedge \tao} (f(X_s)-\lambda)ds +  \ind{\tau< \tau_{\cal{O}}} v(X_\tau)\right\}.
\end{equation}
Then also $Mw_{\alpha_n}$ converges uniformly to $Mw$, which means that $v(x)=Mw(x)$ and therefore $w$ is a solution to \eqref{eq6} and $w\in C(E)$. Since $\inf_{x\in U} w_\alpha(x)=0$ we also have that $\inf_{x\in U} w(x)=0$.

Iterating \eqref{eq6} we obtain inductively for each positive integer $n$
\begin{eqnarray}\label{eq61}
&&{w}(x)=\inf_{V^n} \ee_x^{V^n}\left\{\int_0^{\tao\wedge \tau_n}(f(Y_s)-\lambda) ds + \right. \nonumber \\
&& \left. \sum_{i=1}^n\ind{\tau_i < \tao}c(x_{\tau_i}^i,\xi_i)+\ind{\tau_{n} < \tao}{w}(\tau_n,\xi_n)\right\},
\end{eqnarray}
where $V^n$ denotes impulse strategy consisting of at most $n$ impulses. We restrict here practically to impulses which at time $\tau_{i+1}$  depend only on the behavior of the controlled process after time $\tau_i$. Using the procedure as in section 7 of \cite{BasSte2018} we can show that extension to general impulse strategies (depending on the whole history) does not change the value of $w$.
Then similarly as in the proof of Lemma \ref{lem1} for each $n$ we have
\begin{equation}\label{eq62}
\lambda = \inf_{V^n} { -w(x)+\ee_x^{V^n}\left\{\int_0^{\tao\wedge \tau_n}f(Y_s) ds +  \sum_{i=1}^n\ind{\tau_i < \tao}c(x_{\tau_i}^i,\xi_i)+\ind{\tau_{n} < \tao}{w}(\xi_n)\right\}\over \ee_x^{V^n}\left\{\tao\wedge \tau_n\right\}}.
\end{equation}
Let for each $\vep>0$ the strategy $V_\vep^n$ be an $\vep$ optimal in \eqref{eq62} i.e.
\begin{equation}
\lambda+\vep \geq { -w(x)+\ee_x^{V^n}\left\{\int_0^{\tao\wedge \tau_n}f(Y_s) ds +  \sum_{i=1}^n\ind{\tau_i < \tao}c(x_{\tau_i}^i,\xi_i)+\ind{\tau_{n} < \tao}{w}(\xi_n)\right\}\over \ee_x^{V^n}\left\{\tao\wedge \tau_n\right\}}.
\end{equation}
Then either
\begin{itemize}
\item{(i)} $\liminf_{n\to \infty} \ee_x^{V_\vep^n}\left\{\tao\wedge \tau_n\right\}=\infty$ or
\item{(ii)} $\liminf_{n\to \infty} \ee_x^{V_\vep^n}\left\{\tao\wedge \tau_n\right\}<\infty$.
\end{itemize}
In the case $(i)$ taking into account that $w$ is bounded we have
\begin{equation}\label{eq63}
\lambda+\vep \geq \liminf_{n\to \infty} {\ee_x^{V_\vep^n}\left\{\int_0^{\tao\wedge \tau_n}f(Y_s) ds +  \sum_{i=1}^n\ind{\tau_i < \tao}c(x_{\tau_i}^i,\xi_i)\right\}\over \ee_x^{V_\vep^n}\left\{\tao\wedge \tau_n\right\}}\geq \lambda.
\end{equation}

Let $N(0,\tau)$ be the number of impulses in the time interval $[0,\tau)$. Then since $w$ is bounded and $c(x,\xi)\geq c>0$ we can restrict ourselves to stopping times $\tau_i$ and strategy $V_\vep^n$  such that for each $n$
\begin{equation}\label{eq64}
\lambda +\vep \geq -\|f\|+ {-w(x)+c\ee_x^{V_\vep^n}\left\{N(0,\tao\wedge \tau_n)\right\} \over \ee_x^{V_\vep^n}\left\{\tao\wedge \tau_n\right\}}
\end{equation}
Then in the case $(ii)$ we have that $\ee_x^{V_\vep^n}\left\{N(0,\tao\wedge \tau_n)\right\}$ for a suitably chosen subsequence $n_k$ is bounded and therefore
\begin{equation}
 \ee_x^{V^{n_k}}\left\{\ind{\tau_{n_k} < \tao}{w}(\xi_{n_k})\right\}\to 0
 \end{equation}
as $k\to \infty$.
Consequently for $\hat{x}$ such that $w(\hat{x})=0$ we have
\begin{equation}\label{eq65}
\lambda+\vep \geq \liminf_{n\to \infty} {\ee_{\hat{x}}^{V_\vep^n}\left\{\int_0^{\tao\wedge \tau_n}f(Y_s) ds +  \sum_{i=1}^n\ind{\tau_i < \tao}c(x_{\tau_i}^i,\xi_i)\right\}\over \ee_{\hat{x}}^{V_\vep^n}\left\{\tao\wedge \tau_n\right\}}\geq \lambda.
\end{equation}
Summarizing  \eqref{eq63} and \eqref{eq65} we obtain \eqref{eq5}. Since $\lambda$ is defined in a unique way in \eqref{eq5}, any subsequence of $\lambda_\alpha$ converges to the same value $\lambda$ equal to \eqref{eq5} therefore we have convergence of $\lambda_\alpha$ to $\lambda$ as $\alpha \to 0$.

\end{proof}

\begin{remark} Notice that it may happen that for an optimal impulse strategy we have $\ee_x^V\left\{\tao\right\}=\infty$ as we can see in the example below.
\end{remark}
\begin{example}
Let $E=[0,\infty)$, we have deterministic movement to the right $X_t=x+t$, $f(x)=(f+x^2)\wedge M$ with $M > \sqrt{c}$, $c(x,\xi)=c>0$ and $U=[0,1]$ and ${\cal{O}}=[0,M)$. An optimal strategy is to make shift to $0$ as soon as we are at $\sqrt{c}$ or above it. Starting from $0$ we reach $\sqrt{c}$ at time $\sqrt{c}$ and $\lambda=f+{c\over 3} + \sqrt{c} + {f\over \sqrt{c}}$. Controlled process under this optimal strategy never exits $\cal{O}$.
\end{example}

\section{Bellman equation}\label{S:main}
Denote by $\om$ an increasing sequence of open balls in $E$ with radius $m$.
In this section we shall need the following set of assumptions:
\begin{enumerate}
\item[(\namedlabel{A1}{A.1})]
 Markov process $(X_t)$ has a unique invariant probability measure $\mu$ and there is a continuous solution $q$ to the additive Poisson equation (APE) associated with $f$ such that for any bounded stopping time $\tau$ we have
\begin{equation}
q(x)=\ee_x\left\{\int_0^\tau (f(X_s)-\mu(f))ds+ q(X_\tau)\right\},
\end{equation}
and additionally $q$ is bounded from above by a constant $K$ and $(q(X_t))$ is uniformly integrable.
\item[(\namedlabel{A2}{A.2})] for each $m$ the stopped semigroup $T_t^{\om}$ transforms $C(E)$ into itself.
\item[(\namedlabel{A3}{A.3})] for each $m$ for the first exit time $\taom$ we have $\sup_{x\in E}\ee_x\left\{(\taom)^2\right\}<\infty$,
\item[(\namedlabel{A4}{A.4})] for each $T>0$ we have that $\prob_x\left\{\taom\leq T\right\}$ converges to $0$ uniformly in $x$ from compact sets, as $m\to \infty$.
\end{enumerate}
\begin{remark}
When there is a bounded on compact subset function $K$ and $\gamma>0$ such that for $f\in C(E)$ we have $|\ee_x\left\{f(X_t)\right\}-\mu(f)|\leq \|f\| K(x) e^{-\gamma t}$ then by Lemma 2.2 of \cite{PalSte2017} there is a continuous solution to (APE) and uniform integrability of $q(X_t)$ corresponds to uniform integrability of $K(X_t)$. Then
\[z_t:= \int_0^t (f(X_s)- \mu(f))ds + q(X_t)\]
 is a uniformly integrable martingale. Let  $\Gamma=\left\{x: f(x)\geq \mu(f)\right\}$ be a compact set and suppose that for each $x\in E$ we have $\ee_x\left\{T_\Gamma\right\}<\infty$, where $T_\Gamma=\inf\left\{s\geq 0: X_s\in \Gamma \right\}$. Then for $x\notin \Gamma$
\begin{equation}
q(x)= \ee_x\left\{\int_0^{T_\Gamma} (f(X_s)- \mu(f))ds + q(X_{T_\Gamma})\right\}\leq \ee_x\left\{q(X_{T_\Gamma})\right\}
\end{equation}
and by continuity $q$ is bounded from above.
More general sufficient conditions for existence of solutions to (APE) are formulated in Theorem 2.3 of \cite{Dev2020}. Sufficient conditions for \eqref{A1} were studied in Lemma 2.2 and Lemma 3.18 in \cite{PalSte2017}.
\end{remark}
Let following Theorem \ref{thm1}
\begin{equation}\label{eq9}
\lambda_{(m)}=\inf_{x\in U}\inf_V \liminf_{T\to \infty} {\ee_x^V \Big\{ \int_0^{\taom \wedge T} f(Y_s) ds + \sum_{i = 1}^\infty \ind{\tau_i <\tau_{\taom \wedge T}} c(x_{\tau_i}^i, \xi_i) \Big\}\over \ee_x^V \Big\{\taom \wedge T\Big\}},
\end{equation}
\begin{equation}\label{eq9p}
w^m(x)=\inf_\tau \ee_x\left\{\int_0^{\taom \wedge \tau}(f(X_s)-\lambda_{(m)}) ds + \ind{\tau <\taom}Mw^m(x_\tau)\right\}.
\end{equation}
and $\inf_{\xi \in U}w^m(\xi)= 0$.
\begin{remark}
There is a usual problem how to understand
\[\inf_\tau \ee_x\left\{\int_0^\tau g(X_s)ds + h(X_\tau)\right\}\]
for $g, h \in C(E)$.
We can consider it as
\[\liminf_{T\to \infty} \ee_x\left\{\int_0^{\tau\wedge T} g(X_s)ds + h(X_{\tau\wedge T})\right\},\]
 which coincides with $\ee_x\left\{\int_0^\tau g(X_s)ds + h(X_\tau)\right\}$ with random non necessarily bounded $\tau$, when we  assume that $\ee_x\left\{\tau\right\}<\infty$ and take into account quasi leftcontinuity of $(X_t)$ (see Theorem 3.13 of \cite{Dyn1965}). Consequently if we know that for optimal stopping times we can restrict ourselves to stopping times $\tau$ such that $\ee_x\left\{\tau\right\}\leq M$ then   $\inf_\tau \ee_x\left\{\int_0^\tau g(X_s)ds + h(X_\tau)\right\}$
is the same as the infimum over all bounded stopping times.
\end{remark}
We have the following main result
\begin{theorem}\label{thm2} Under \eqref{A1}-\eqref{A4}  for every convergent subsequence of $\lambda_{(m_k)}$ such that $\lim \lambda_{(m_k)}=\lambda<\mu(f)$ we have that
\begin{equation}\label{form1}
\lambda=\inf_V {J}\big(x, V\big)
\end{equation}
and there is a continuous  bounded function $w$ such that $\inf_{\xi\in U} w(\xi)=0$ and
\begin{equation}\label{form2}
w(x)=\inf_\tau \ee_x\left\{\int_0^\tau (f(X_s)-\lambda)ds + Mw(X_\tau)\right\}.
\end{equation}
Furthermore \eqref{form2} defines an optimal strategy $\hat{V}=(\hat{\tau}_i, \hat{\xi}_i)$ for $J$:

$\hat{\tau}_1=\inf\left\{s\geq 0: w(Y_s)=Mw(Y_s)\right\}$,
$\hat{\tau}_{n+1}=\hat{\tau}_n + \hat{\tau}_1 \circ \theta_{\hat{\tau}_n}$, where $\theta_t$ is a shift operator corresponding to the controlled process $(Y_t)$ and $\hat{\xi}_n=\argmin_{z\in U}(c(x_{\hat{\tau}_n}^n,z)+w(z)$.
\end{theorem}
\begin{proof}
Let $w^m$ be a solution to Bellman equation \eqref{eq9p} (which exists by Theorem \ref{thm1}).  Assume that $\lambda_{(m_k)}\to \lambda\leq \mu(f)-a$, with $a>0$, as $k\to \infty$. Since $M{w}^{m_k}(x)\leq \|c\|+\inf_{z\in U}{w}^{m_k}(z)=\|c\|$ and $M{w}^{m_k}(x)\geq c+\inf_{z\in U}{w}^{m_k}(z)=c>0$ as well as  $|M{w}^{m_k}(x)-M{w}^{m_k}(y)|\leq \sup_{\xi \in U}|c(x,\xi)-c(y,\xi)|$ the family of functions  $\left(M{w}^{m_k}\right)$ is bounded and uniformly continuous. Therefore there is a further subsequence, for simplicity still denoted by $m_k$ and a function $v\in C(E)$ such $M{w}^{m_k}$ converges uniformly on compact subsets to $v$ such that  $c\leq v\leq \|c\|$.  By \eqref{A1} and the fact that $M{w}^m\geq 0$
\begin{eqnarray}\label{eq8}
&&{w}^m(x)=\inf_\tau \ee_x\left\{(\taom\wedge \tau)(\mu(f)-\lambda_{(m)})  - q(X_{\taom\wedge \tau}) +  \right. \nonumber \\
&& \left. \ind{\tau < \taom}M{w}^m(x_\tau) \right\}+q(x) \geq
 q(x)+ \nonumber \\
&&  \inf_\tau \ee_x\left\{(\taom\wedge \tau)(\mu(f)-\lambda_{(m)})- q(X_{\taom\wedge \tau})\right\},
\end{eqnarray}
so that assuming that $\lambda_m<\mu(f)$, taking into account that ${w}^m(x)\leq M{w}^m(x)\leq \|c\|$ and $q(\cdot)\leq K$, in infimum in \eqref{eq8} we can restrict ourselves to stopping times $\tau$ such that
\begin{equation}
\ee_x\left\{(\taom\wedge \tau)\right\}\leq {\|c\|-q(x)+K \over \mu(f)-\lambda_m}.
\end{equation}
Consequently for a sufficiently large $k$, say $k\geq k_0$ we see that we can restrict ourselves in $w^m$ to stopping times $\tau$ such that
\begin{equation}\label{eq10}
\ee_x\left\{\taomk\wedge \tau\right\}\leq {\|c\|-q(x)+K \over a}.
\end{equation}
Let
\begin{equation}\label{eq10p}
z_m(x)=\inf_\tau \ee_x\left\{\int_0^{\taom\wedge \tau}(f(X_s)-\lambda)ds + v(X_{\taom\wedge \tau})\right\}.
\end{equation}
By the same arguments as above we can restrict in \eqref{eq10p} to stopping times such that \eqref{eq10} is satisfied.
Clearly $z_m(x)$ is a nonincreasing sequence and let $z(x):=\lim_{m\to \infty}z_m(x)$. We are going to show that $z(x)$ is of the form
\begin{equation}\label{eq11}
z(x)=\inf_\tau \ee_x\left\{\int_0^{\tau}(f(X_s)-\lambda)ds + v(X_{\tau})\right\}.
\end{equation}
We also can restrict in \eqref{eq11} to stopping times such that
\begin{equation}\label{eq11p}
\ee_x\left\{\tau\right\}\leq {\|c\|-q(x)+K \over a}.
\end{equation}
Therefore there is an $\epsilon$-optimal stopping time $\tau^*_\epsilon$ for \eqref{eq11} such that for some deterministic $T$ we have  $\tau^*_\epsilon\leq T$, $\prob_x$ a.s..
Then
\begin{eqnarray}
&&0\leq z_m(x)-z(x)\leq \ee_x\left\{\int_0^{\taom\wedge {\tau^*_\epsilon}}(f(X_s)-\lambda)ds + v(X_{\taom\wedge {\tau^*_\epsilon}})\right\}- \nonumber \\
&&\ee_x\left\{\int_0^{{\tau^*_\epsilon}}(f(X_s)-\lambda)ds + v(X_{{\tau^*_\epsilon}})\right\}+\epsilon \leq \nonumber \\
&&\ee_x\left\{{\tau^*_\epsilon}-\taom\wedge {\tau^*_\epsilon}\right\} \|f-\lambda\| +
\|c\|\prob_x\left\{\taom\leq T\right\} +\epsilon\nonumber \\
&& \leq \prob_x\left\{\taom\leq T\right\} T \|f-\lambda\|+ \|c\|\prob_x\left\{\taom\leq T\right\} +\epsilon,
\end{eqnarray}
which by \eqref{A4} converges to $\epsilon$. Consequently $z$ is of the form \eqref{eq11}.
By Theorem \ref{thma2} we have that $z\in C(E)$ and since by Proposition \ref{optnd} functions $z_m$ are continuous, using Dini's lemma we have that $z_m$ converges uniformly on compact subsets to $z$.
We shall evaluate now the difference between ${w}^{m_k}(x)$ and $z_{m_k}(x)$.
Using \eqref{A1} we see that we can restrict ourselves in $z_{m_k}$ to stopping times $\tau$ such that \eqref{eq10} is satisfied.
For such stopping times we have for $k\geq k_0$
\begin{eqnarray}\label{eq12}
&&|\ee_x\left\{\int_0^{\taom\wedge \tau}(f(X_s)-\lambda_{(m_k)})ds\right\}-\ee_x\left\{\int_0^{\taom\wedge \tau}(f(X_s)-\lambda)ds\right\}|\leq \nonumber \\
&&|\ee_x\left\{(\taom\wedge \tau)\right\}|\lambda_{(m_k)}-\lambda|\leq {\|c\|-q(x)+K \over a}|\lambda_{(m_k)}-\lambda|.
\end{eqnarray}
For stopping times $\tau$ satisfying \eqref{eq10} we also have (recall that $c\leq M\bar{w}^{m_k}\leq \|c\|$)
\begin{eqnarray}\label{eq13}
&&|\ee_x\left\{\ind{\tau < \taomk} M{w}^{m_k}(X_\tau)\right\}-\ee_x\left\{ v(X_{\taomk\wedge \tau})\right\}|\nonumber \\
&&\leq \ee_x\left\{\ind{\tau\geq\taomk}\right\}\|c\|  + \nonumber \\
&& |\ee_x\left\{M{w}^{m_k}(X_{\taomk \wedge \tau} ) - v(X_{\taomk\wedge \tau})\right\}|\leq \nonumber \\
&& \ee_x\left\{\ind{\tau\geq\taomk}\ind{\taomk\leq T}\right\}\|c\|+
\ee_x\left\{\ind{\tau\geq\taomk}\ind{\taomk > T}\right\}\|c\| +\nonumber \\
&&|\ee_x\left\{\ind{\taomk\wedge \tau \leq T}(M{w}^{m_k}(X_{\taomk \wedge \tau}) - v(X_{\taomk\wedge \tau}))\right\}| + \nonumber \\
&&\ee_x\left\{\ind{\taomk\wedge \tau \geq T}\|c\|\right\} \leq \ee_x\left\{\ind{\taomk\leq T}\right\}\|c\|+ \nonumber \\
&& \ee_x\left\{\ind{\taomk\wedge \tau \geq T}\right\}2\|c\|+  \nonumber \\
&& |\ee_x\left\{\ind{\taomk\wedge \tau \leq T}\ind{\rho(x,X_{\taomk \wedge \tau})\leq R} (M{w}^{m_k}(X_{\taomk \wedge \tau}) - \right. \nonumber \\
&& \left.  v(X_{\taomk\wedge \tau})\right\})|+ \ee_x\left\{\ind{\taomk\wedge \tau \leq T}\ind{\rho(x,X_{\taomk \wedge \tau})\geq R}\right\}\|c\| \nonumber \\
&& =b_k+e_k+g_k+h_k,
\end{eqnarray}
where $e_k\leq {\ee_x\left\{\taomk\wedge \tau\right\}\|c\| \over T}$ and $b_k$ converges to $0$ uniformly on compact sets by \eqref{A4}.
For $x$ from a compact set $K$ and given $\vep>0$ and $T>0$ one can find by Proposition 2.1 of  \cite{PalSte2010} $R>0$ such that
\begin{equation}\label{eq13'}
\sup_{x\in K}\prob_x\left\{\exists_{t\leq T} \ \rho(x,X(t))\geq R\right\}\leq \vep,
\end{equation}
where $\rho$ is the metric on $E$.
Therefore for $T$ sufficiently large $e_k\leq \epsilon$ for all $k$ and $x\in K$ and for fixed $T$ for large enough $k$ we have that $b_k\leq \epsilon$ for $x$ from $K$. From \eqref{eq13'} we have that $h_k\leq \vep \|c\|$. Finally for $x\in K$ there is $k_1$ such that for $k\geq k_1$
on the set $\left\{\taomk\wedge \tau \leq T\right\}\cap  \left\{\rho(x,X_{\taomk \wedge \tau})\leq R\right\}$ we have
$|M{w}^{m_k}(X_{\taomk \wedge \tau}) - v(X_{\taomk\wedge \tau})|\leq \vep$. Therefore for $x\in K$ and sufficiently large $k$ we can make $b_k+e_k+g_k+h_k$ arbitrarily small.

 Summarizing \eqref{eq12} and \eqref{eq13} we obtain that $|{w}^{m_k}(x)-z_{m_k}(x)|$ converges to $0$ uniformly on compact sets. Consequently ${w}^{m_k}(x)$ converges to $z(x)$ uniformly on compact sets and additionally $\inf_{\xi \in U} z(\xi)=0$.  Therefore $M{w}_{m_k}(x)$ converges uniformly to $Mz(x)=v(x)$ and
\begin{equation}\label{eq14}
z(x)=\inf_\tau \ee_x\left\{\int_0^{\tau}(f(X_s)-\lambda)ds + Mz(X_{\tau})\right\},
\end{equation}
which completes the proof of \eqref{form2} with $z\equiv w$.
Iterating the last formula we obtain inductively for each positive integer $n$
\begin{equation}\label{eq15}
{w}(x)=\inf_{V^n} \ee_x^{V^n}\left\{\int_0^{\tau_n}(f(Y_s)-\lambda) ds + \sum_{i=1}^n c(x_{\tau_i}^i,\xi_i)+{w}(\xi_n)\right\}.
\end{equation}
where $V^n$ denotes impulse strategy consisting of at most $n$ impulses. In fact, we have first \eqref{eq15} for so called shifted impulse strategies (see the definition in section 2 of \cite{BasSte2018}), which then can be extended using technics of section 7 of \cite{BasSte2018} to any impulse strategies consisting of $n$ impulses.   Notice that we can restrict ourselves to stopping times such that \eqref{eq11p} is satisfied.
Then
\begin{equation}
\lambda=\inf_{V^n} {1\over \ee_x^{V^n}\left\{\tau_n\right\}} \ee_x^{V^n}\left\{\int_0^{\tau_n} f(Y_s)ds + \sum_{i=1}^n c(x_{\tau_i}^i,\xi_i)+{w}(\xi_n)\right\}
\end{equation}
and then for any impulse strategy $V$ (since $w$ is bounded on $U$)
\begin{equation}
\lambda\leq \liminf_{n\to \infty} {1\over \ee_x^{V}\left\{\tau_n\right\}} \ee_x^{V}\left\{\int_0^{\tau_n} f(Y_s)ds + \sum_{i=1}^n c(x_{\tau_i}^i,\xi_i)\right\},
\end{equation}
with equality for an optimal impulse strategy $\hat{V}$ defined by \eqref{form2}. This completes the proof of \eqref{form1}.
\end{proof}

\begin{remark} Note first that by \eqref{A4} we have that
\begin{equation}\label{limit} \ee_x \Big\{\taom\Big\}\to \infty
\end{equation}
 as $m\to \infty$. In fact, $\ee_x \Big\{\taom\Big\}\geq T \prob_x\left\{\taom\geq T\right\}$. Letting $m\to \infty$, using \eqref{A4} we obtain $\liminf_{m\to \infty} \ee_x \Big\{\taom\Big\}\geq T$ for any $T>0$ form which \eqref{limit} follows. Then for no impulse strategy we have using \eqref{A1}
 \begin{equation}
 {\ee_x \Big\{\int_0^{\taom} f(X_s)ds\Big\}\over {\ee_x \Big\{\taom \Big\}}}={q(x)-\ee_x \left\{q(X_{\taom})\right\} \over {\ee_x \Big\{\taom \Big\}}} +\mu(f).
 \end{equation}
Letting $m\to \infty$, taking into account that $q$ is bounded from above we obtain
\begin{equation}\label{limitt}
\liminf_{m\to \infty}{\ee_x \Big\{\int_0^{\taom} f(X_s)ds\Big\}\over {\ee_x \Big\{\taom \Big\}}}\geq \mu(f)
\end{equation}
When $q$ is bounded we have equality in \eqref{limitt} and in Theorem \ref{thm2} we can write $\lim_{m\to \infty} \lambda_m=\lambda \leq \mu(f)$, and when $\lambda <\mu(f)$ we have \eqref{form1} and \eqref{form2}.
\end{remark}

We can now formulate a result concerning the functional \eqref{eqn:functional}.
\begin{proposition}
Under the assumptions of Theorem \ref{thm2} for any impulse strategy $V$  we have that
\begin{equation}\label{eq16}
\lambda\leq \liminf_{T\to \infty}  \frac{1}{T} \ee_x^V \Big\{ \int_0^{T} f(Y_s) ds + \sum_{i = 1}^\infty \ind{\tau_i\leq T} c(x_{\tau_i}^i, \xi_i) \Big\}.
\end{equation}
If for the strategy $\hat{V}$ defined in Theorem \ref{thm2} we have that
$\lim_{T\to \infty} {1 \over  T}\ee_x^{\hat{V}}\left\{w(Y_T)\right\}=0$
then
\begin{equation}\label{eq16'}
\lambda=\inf_V \hat{J}(x,V)=\hat{J}(x,\hat{V}).
\end{equation}
\end{proposition}
\begin{proof}
Iterating \eqref{form2} for any $T>0$ we have
\begin{equation}
w(x)\leq \ee_x^V\left\{ \int_0^{\tau_n\wedge T}(f(Y_s)-\lambda)ds + \sum_{i=1}^n \ind{\tau_i\leq T} c(x_{\tau_i}^i,\xi_i)+w(\tau_n\wedge Y_T)\right\}
\end{equation}
with equality for the strategy $\hat{V}$. Taking into account that any nearly optimal strategy $V$ consists of at most  a finite number of impulses on the time interval $[0,T]$ we obtain for any such strategy $V$
\begin{equation}
w(x)\leq \ee_x^V\left\{ \int_0^T(f(Y_s)-\lambda)ds + \sum_{i=1}^\infty \ind{\tau_i\leq T} c(x_{\tau_i}^i,\xi_i)+w(Y_T)\right\}
\end{equation}
with equality for the strategy $\hat{V}$.
Then since $w(Y_T)\leq \|c\|$ we have
\begin{eqnarray}\label{eq16''}
&& \lambda \leq \liminf_{T\to \infty}{1 \over T} \ee_x^V\left\{ \int_0^T f(Y_s)ds + \sum_{i=1}^\infty \ind{\tau_i\leq T} c(x_{\tau_i}^i,\xi_i)+w(Y_T)\right\} \nonumber \\
&& \leq \liminf_{T\to \infty} {1 \over T} \ee_x^V\left\{ \int_0^T f(Y_s)ds + \sum_{i=1}^\infty \ind{\tau_i\leq T} c(x_{\tau_i}^i,\xi_i)\right\}
\end{eqnarray}
and \eqref{eq16} follows. For the strategy $\hat{V}$ when  $\lim_{T\to \infty} {1 \over  T}\ee_x^{\hat{V}}\left\{w(Y_T)\right\}=0$ we obtain equality in \eqref{eq16''}, which completes the proof of \eqref{eq16}.
\end{proof}

\begin{corollary}
Under the assumptions of Theorem \ref{thm2} assuming additionally that  the set $G=\left\{x: f(x)\leq \lambda\right\}$ is compact we have that the function $w$ is bounded and consequently we have \eqref{eq16'}.
\end{corollary}
\begin{proof}
We use an analog of \eqref{aeq3}, namely
\begin{equation}
w(x)=\inf_\tau \ee_x\left\{\int_0^{\tau \wedge T_{G}}(f(X_s)-\lambda)ds + \ind{\tau<T_G}Mw(X_\tau) + \ind{T_G\leq \tau}w(X_{T_G})\right\}.
\end{equation}
Then since $Mw(x)\geq c>0$
\begin{equation}
w(x)\geq c+ \inf_\tau \ee_x\left\{\ind{T_G\leq \tau}w(X_{T_G})\right\},
\end{equation}
which is bounded from below by continuity of $w$ and compactness of $G$. Since $w\leq \|c\|$ function $w$ is bounded.
\end{proof}

\section{Appendix}\label{S:app}
We recall Theorem 4.3 of \cite{PalSte2011} (formulated there for supremum of stopping times)
\begin{theorem}\label{optd}
Assume that $f,G,H\in C(E)$, $G\geq H$ and $\alpha>0$. Under the assumptions that $T_t C_0(E)\subset C_0(E)$, $T_t^{\cal{O}}C(E)\subset C(E)$ we have that
\begin{eqnarray}
w_\alpha(x)&:=& \inf_\tau \ee_x\left\{\int_0^{\tau \wedge \tau_{\cal{O}}}e^{-\alpha s}f(X_s)ds + \ind{\tau<\tau_{\cal{O}}}e^{-\alpha \tau}G(X_\tau)+ \right. \nonumber\\
&& \left. \ind{\tau\geq\tau_{\cal{O}}}e^{-\alpha \tau_{\cal{O}}}H(X_{\tau_{\cal{O}}})\right\}
\end{eqnarray}
is continuous and bounded and an optimal stopping time $\tau^*_\alpha$ is given by the formula
\begin{equation}
\tau^*_\alpha=\inf\left\{s\geq 0: w_\alpha(X_s)\geq G(X_s) \ \  or \ \ X_s\notin \cal{O}\right\}.
\end{equation}
\end{theorem}
For an undiscounted case we have
\begin{proposition}\label{optnd}
Under the assumptions of Theorem \ref{optd} and $\sup_{x\in E}\ee_x\left\{(\tao)^2\right\}<\infty$
we have that $w_\alpha$ converges uniformly to $w$, as $\alpha \to 0$, where
\begin{equation}
w(x):=\inf_\tau \ee_x\left\{\int_0^{\tau \wedge \tao}f(X_s)ds + \ind{\tau<\tau_{\cal{O}}}G(X_\tau)+  \ind{\tau\geq\tau_{\cal{O}}}H(X_{\tau_{\cal{O}}})\right\}
\end{equation}
is continuous bounded and an optimal stopping time $\tau^*$ is given by the formula
\begin{equation}
\tau^*=\inf\left\{s\geq 0: w(X_s)\geq G(X_s) \ \  or \ \ X_s\notin \cal{O}\right\}.
\end{equation}
\end{proposition}
\begin{proof}
Since $1-e^{-\alpha s}\leq \alpha s$, we have
\begin{eqnarray}
|w(x)-w_\alpha(x)|&\leq& \ee_x\left\{\int_0^{\tao} (1-e^{-\alpha s})\|f\|ds + (1-e^{-\alpha \tao})(\|G\|\vee \|H\|)\leq \right. \nonumber \\
&& \left. {\alpha \over 2} \ee_x\left\{(\tao)^2\right\}\|f\| + \alpha \ee_x\left\{\tao\right\}(\|G\|\vee \|H\|) \right\}
\end{eqnarray}
from which the convergence and continuity of $w$ follows. To show optimality of $\tau^*$ recall so called penalty equation corresponding to $w_\alpha$, see (4) of \cite{PalSte2011}:
\begin{equation}
w_\alpha^\beta(x)=\ee_x\left\{\int_0^{\tao} e^{-\alpha s} \left(f(X_s)+\beta(w_\alpha^\beta(X_s)-G(X_s))^+\right)ds + e^{-\alpha \tao}H(X_{\tao})\right\}.
\end{equation}
By Lemma 2.2 of \cite{PalSte2011} such function $w_\alpha^\beta$ exists and by Proposition 2.9 of \cite{PalSte2011} it decreases to $w_\alpha$. For any stopping time $\tau$ we can write
\begin{eqnarray}\label{ap1}
w_\alpha^\beta(x)&=&\ee_x\left\{\int_0^{\tao\wedge \tau} e^{-\alpha s} \left(f(X_s)+\beta(w_\alpha^\beta(X_s)-G(X_s))^+\right)ds + \ind{\tau<\tao}e^{-\alpha \tau} w_\alpha^\beta(X_\tau) + \right. \nonumber \\
&&\left. \ind{\tao\leq \tau}e^{-\alpha \tao}H(X_{\tao})\right\}.
\end{eqnarray}
By Theorem 4.3 we know that
\begin{equation}
\tau_\alpha^*=\inf\left\{s\geq 0: w_\alpha(X_s)\geq G(X_s) \ or \ X_s\notin \cal{O}\right\}
\end{equation}
is an optimal time for $w_\alpha$. Let  $\tau^\epsilon=\inf\left\{s\geq 0: w(X_s)+\epsilon\geq G(X_s) \ or \ X_s\notin \cal{O}\right\}$. When $\|w_\alpha-w\|\leq \epsilon$ we have that whenever $w_\alpha(X_s)\geq G(X_s)$ then also $w(X_s)+\epsilon\geq G(X_s)$ which means that $\tau^\epsilon\leq \tau_\alpha^*$. Consequently from \eqref{ap1}
we obtain
\begin{equation}
w_\alpha^\beta(x)=\ee_x\left\{\int_0^{\tao\wedge \tau^\epsilon} e^{-\alpha s} f(X_s)ds + \ind{\tau^\epsilon<\tao}e^{-\alpha \tau^\epsilon} w_\alpha^\beta(X_{\tau^\epsilon}) + \ind{\tao \leq \tau^\epsilon}e^{-\alpha \tao}H(X_{\tao})\right\}.
\end{equation}
Letting in the last equation $\beta \to \infty$ and then $\alpha \to 0$ we finally obtain
\begin{equation}
w(x)=\ee_x\left\{\int_0^{\tao\wedge \tau^\epsilon} f(X_s)ds + \ind{\tau^\epsilon<\tao} w(X_{\tau^\epsilon}) + \ind{\tao \leq \tau^\epsilon}H(X_{\tao})\right\}.
\end{equation}
It remains now to notice that $\tau^\epsilon$ with $\epsilon \to 0$ form an increasing sequence  of  stopping times which by quasileftcontinuity of $(X_s)$ (see Theorem 3.13 of \cite{Dyn1965}) converges to $\tau^*$ and therefore by continuity of $w$ we obtain
\begin{equation}
w(x)=\ee_x\left\{\int_0^{\tao\wedge \tau^*} f(X_s)ds + \ind{\tau^*<\tao} w(X_{\tau^*}) + \ind{\tao \leq \tau^*}H(X_{\tao})\right\}.
\end{equation}
completing this way the proof of optimality of $\tau^*$.
\end{proof}

We finally formulate an ergodic stopping result
\begin{theorem}\label{thma2}
Under \eqref{A1} let $f, w\in C(E)$ and $\mu(f)>0$. Then
\begin{equation}\label{aeq1}
z(x):=\inf_\tau \ee_x\left\{\int_0^{\tau}f(X_s)ds + w(X_{\tau})\right\}
\end{equation}
is continuous and an optimal stopping time $\tau^*$ is of the form
\begin{equation}\label{aeq2}
\tau^*=\inf\left\{s\geq 0: z(X_s)\geq w(X_s)\right\}.
\end{equation}
Furthermore for any stopping time $\sigma$ we have the following version of Bellman equation
\begin{equation}\label{aeq3}
z(x)=\inf_\tau \ee_x\left\{\int_0^{\tau\wedge \sigma}f(X_s)ds + \ind{\tau<\sigma} w(X_{\tau})+\ind{\sigma \leq \tau} z(X_\sigma)\right\}.
\end{equation}
\end{theorem}
\begin{proof}
By \eqref{A1} we can write
\begin{equation}\label{ap2}
z(x)=q(x)+\inf_\tau \ee_x\left\{\mu(f)\tau+w(X_\tau)-q(X_\tau)\right\}.
\end{equation}
Since $q(x)\leq K$ for $x\in E$, to show continuity of $z$ it suffices to prove continuity of the function
\begin{equation}\label{ap3}
v(x)=\inf_{\tau} \ee_x\left\{\mu(f)\tau+g(X_\tau)\right\},
\end{equation}
where $g(x)=-q(x)+K+w(x)+\|w\|\geq 0$.
To show continuity of $v$ we consider discrete time problem
\begin{equation}\label{ap4}
v_\delta(x)=\inf_{\tau\in {\cal T}_\delta} \ee_x\left\{\mu(f)\tau+g(X_\tau)\right\},
\end{equation}
where ${\cal T}_\delta$ is a family of stopping times taking values in $\left\{0,\delta,\ldots,n\delta, \ldots\right\}$.
Consider Bellman equation corresponding to \eqref{ap4}. We are looking for a function $r$ such that
\begin{equation}\label{ap5}
r(x)=(\mu(f)\delta+P_\delta r(x))\wedge g(x):=T_\delta r(x).
\end{equation}
Notice that by \eqref{A1} function $q$ is continuous and for each $n$ also $P_\delta^n q$ is continuous.
Consider two sequences $(\underbar{r}_n)$, $(\bar{r}_n)$ of functions approximating solutions to \eqref{ap5} defined as follows
$\underbar{r}_1(x)=T_\delta (0)(x)$, $\underbar{r}_2(x)=T_\delta \underbar{r}_1(x)$, $\bar{r}_1(x)=T_\delta g(x)$, $\bar{r}_2(x)=T_\delta \bar{r}_1(x)$ and inductively  $\underbar{r}_{n+1}(x)=T_\delta \underbar{r}_n(x)$ and
$\bar{r}_{n+1}(x)=T_\delta \bar{r}_n(x)$ for positive integer $n$. Now, $\underbar{r}_1(x)\geq 0$ and $\bar{r}_1(x)\leq g(x)$ and by monotonicity of $T_\delta$, we have  $\underbar{r}_{n+1}(x) \geq  \underbar{r}_n(x)$ and  $\bar{r}_{n+1}(x)\leq T_\delta \bar{r}_n(x)$
so that $\underbar{r}_{n}(x)$ is a nondecreasing sequence while  $\bar{r}_{n}(x)$ is a nonincreasing sequence of continuous functions. Let $\underbar{r}(x)=\lim_{n\to \infty} \underbar{r}_{n}(x)$ and $\bar{r}(x)=\lim_{n\to \infty} \bar{r}_{n}(x)$. Clearly $\underbar{r}$ is lower semicontinuous while $\bar{r}$ is upper semicontinuous. Moreover one can show that
\begin{equation}\label{ap6}
\underbar{r}_n(x)=\inf_{\tau \in {\cal T}_\delta} \ee_x\left\{\mu(f)(\tau\wedge n \delta)+\ind{\tau<n \delta}g(X_\tau)\right\}
\end{equation}
and
\begin{equation}\label{ap7}
\bar{r}_n(x)= \inf_{\tau \in {\cal T}_\delta} \ee_x\left\{\mu(f)(\tau \wedge n \delta)+g(X_{\tau \wedge n\delta})\right\}.
\end{equation}
In both problems \eqref{ap6} and \eqref{ap7} we can restrict ourselves to stopping times $\tau$ such that
\begin{equation}\label{ap8}
\ee_x\left\{\tau \wedge n \delta \right\}\leq {g(x) \over \mu(f)},
\end{equation}
so that letting $n\to \infty$, by Fatou lemma we have
\begin{equation}\label{ap9}
\ee_x\left\{\tau \right\}\leq {g(x) \over \mu(f)}.
\end{equation}
Therefore choosing $\vep$-optimal stopping times for $\underbar{r}_n$ satisfying \eqref{ap9} we have
\begin{equation}\label{ap10}
0\leq \bar{r}_n(x)- \underbar{r}_n(x)\leq \vep + \ee_x\left\{\ind{\tau \geq n\delta} g(X_{n\delta})\right\}
\end{equation}
and since by \eqref{ap9} $\prob_x\left\{\tau \geq n\delta\right\} \leq {g(x)\over \mu(f)n\delta}$ using uniform integrability of $q(X_{n\delta})$ and therefore also of $g(X_{n\delta})$ we obtain that $\underbar{r}(x)=\bar{r}(x)$ is continuous and coincides with $v_{\delta}(x)$.
Let now $\tau$ a $\vep$-optimal stopping time for $v(x)$. We can restrict ourselves to stopping time $\tau$ such that \eqref{ap9} is satisfied. Let

$\tau_\delta=\inf\left\{(n+1)\delta: n\delta < \tau \leq (n+1)\delta, n=0,1,\ldots \right\}$ and $\tau_\delta=0$, whenever $\tau=0$. Then
\begin{eqnarray}\label{ap11}
&&0\leq v_\delta(x)-v(x)\leq \vep + \mu(f)\delta + \nonumber \\
&&\ee_x\left\{w(X_{\tau_\delta})-w(X_\tau) + q(X_\tau)-q(X_{\tau_\delta})\right\} \leq \nonumber \\
&& \vep + (\mu(f)+2\|f\|)\delta + \ee_x\left\{\ind{\tau\leq T} (w(X_{\tau_\delta})-w(X_\tau)) + \ind{\tau > T} 2\|w\|\right\}
\end{eqnarray}
since by the definition of $q$ we have that $\ee_x\left\{q(X_\tau)-q(X_{\tau_\delta})\right\}\leq \delta 2\|f\|$.
For $x$ from a compact set $K$ and given $\vep>0$ one can find by Proposition 2.1 of  \cite{PalSte2010} $R>0$ such that (see also \eqref{eq13'})
\begin{equation}\label{ap12}
\sup_{x\in K}\prob_x\left\{\exists_{t\leq T} \ \rho(x,X(t))\geq R\right\}\leq \vep,
\end{equation}
where $\rho$ is the metric on $E$. Furthermore for a given compact set $B(K,R)=\left\{y: \rho(y,K)\leq R\right\}$ for a given $\vep>0$ and $\gamma>0$ by Proposition 6.4 of \cite{BasSte2018} there is $\kappa>0$ such that
\begin{equation}\label{ap13}
\sup_{x\in B(K,R)} \prob_x\left\{\exists_{t\leq \kappa} \ \rho(x,X(t))\geq \gamma\right\}\leq \vep.
\end{equation}
With the use of \eqref{ap12} and \eqref{ap13} we can evaluate for $\delta\leq \kappa$
\begin{eqnarray}\label{ap14}
&&|\ee_x\left\{\ind{\tau\leq T} (w(X_{\tau_\delta})-w(X_\tau))\right\}|\leq \nonumber \\
&&|\ee_x\left\{\ind{\tau\leq T}\ind{X_\tau\in B(K,R)} (w(X_{\tau_\delta})-w(X_\tau))\right\}|+
 2\|w\| \vep \leq \nonumber \\
&& |\ee_x\left\{\ind{\tau\leq T}\ind{X_\tau\in B(K,R)}\ind{\rho(X_\tau,X_{\tau_\delta})\leq \gamma}
 (w(X_{\tau_\delta})-w(X_\tau))\right\}| + 4\|w\| \vep .
\end{eqnarray}
By continuity of $w$ it is uniformly continuous in a compact set so that for a sufficiently small $\gamma$ we have that
\begin{equation}\label{ap15}
|\ee_x\left\{\ind{\tau\leq T}\ind{X_\tau\in B(K,R)}\ind{\rho(X_\tau,X_{\tau_\delta})\leq \gamma}
 (w(X_{\tau_\delta})-w(X_\tau))\right\}|\leq \vep.
\end{equation}
Summarizing \eqref{ap11}, \eqref{ap14} and \eqref{ap15} taking also into account \eqref{ap9}
we finally obtain
\begin{equation}
0\leq v_\delta(x)-v(x)\leq \vep + (\mu(f)+2\|f\|)\delta + 4\|w\| \vep + \vep + 2\|w\|{g(x) \over \mu(f)T},
\end{equation}
so that $v_\delta(x)\to v(x)$, as $\delta\to 0$, uniformly in $x$ from compact sets. Consequently $v$ and then also $z$ is a continuous function. This completes the first part of the proof of Theorem \ref{thma2}.

Consider a finite horizon stopping problem
\begin{equation}\label{ap16}
z_T(x):=\inf_\tau \ee_x\left\{\int_0^{\tau \wedge T}f(X_s)ds + w(X_{\tau\wedge T})\right\}.
\end{equation}
By Corollary 1 of \cite{Ste2011} we have that $z_T\in C(E)$ and therefore an optimal stopping $\tau_T$ time is of the form
\begin{equation}\label{ap17}
\tau_T=\inf\left\{s\geq 0: z_{T-s}(X_s)\geq w(X_s)\right\}.
\end{equation}
Clearly $z_T(x)$ is decreasing to $z(x)$, as $T\to \infty$ and since both functions are continuous,\ by Dini's lemma we have convergence uniform on compact sets. Stopping times $\tau_T$ are increasing in $T$ and $\tau_T\leq \tau^*$.
We know that we can restrict ourselves in \eqref{aeq1} to stopping times $\tau$ such that
\begin{equation}\label{ap18}
\ee_x\left\{\tau\right\}\leq {\|w\|-q(x)+K \over \mu(f)}.
\end{equation}
Clearly $\tau^*$ also satisfies \eqref{ap18}. Furthermore $\tau_T\to \tau^*$ as $T\to \infty$.
Then we have, using quasilefcontinuity of $(X_t)$ (see Theorem 3.13 of \cite{Dyn1965}) and \eqref{ap18}
\begin{equation}\label{ap19}
z_T(x)=\ee_x\left\{\int_0^{\tau_T}f(X_s)ds + w(X_{\tau_T})\right\}\to \ee_x\left\{\int_0^{\tau^*}f(X_s)ds + w(X_{\tau^*})\right\}
\end{equation}
as $T\to \infty$, which means that $\tau^*$ is an optimal stopping time for $z$.

It remains to show \eqref{aeq3}. Notice first that by \eqref{ap5} $m_\delta(n\delta)=\mu(f)n\delta + v_\delta(X_{n\delta})$ is a submartingale and letting $\delta\to 0$ also $\hat{m}(t)=\mu(f)t+v(t)$ is a right continuous submartingale. Consequently also $m(t)=\int_0^t f(X_s)ds + z(X_t)$ is a right continuous submartingale. Therefore for any bounded stopping time $\tau$ and stopping time $\sigma$ we have $\ee_x\left\{m(\tau)|F_{\tau\wedge \sigma}\right\}\geq m(\tau\wedge \sigma)$ and
\begin{eqnarray}\label{ap20}
&&\int_0^{\tau\wedge \sigma} f(X_s)ds + z(X_{\tau \wedge \sigma})\leq \ee_x\left\{\int_0^\tau f(X_s)ds + z(X_{\tau})|F_{\tau\wedge \sigma}\right\}\leq  \nonumber \\
&&\ee_x\left\{\int_0^\tau f(X_s)ds + \ind{\tau < \sigma} z(X_{\tau})+\ind{\sigma\leq\tau}w(X_\tau)|F_{\tau\wedge \sigma}\right\}
\end{eqnarray}
Subtracting from both sides of  $\ind{\tau < \sigma} z(X_{\tau})$ and then adding to both sides
$\ind{\tau < \sigma} w(X_{\tau})$ we obtain that
\begin{equation}\label{ap21}
\ee_x\left\{\int_0^\tau f(X_s)ds + w(X_{\tau})\right\}\geq \ee_x\left\{\int_0^{\tau\wedge \sigma} f(X_s)ds + \ind{\tau < \sigma} w(X_{\tau})+ \ind{\sigma\leq \tau}z(X_\sigma) \right\}.
\end{equation}
Since in the definition of $z$ we may restrict to stopping times $\tau$ such that \eqref{ap18} is satisfied such stopping times can be approximated by bounded stopping times we have
that
\begin{equation}
z(x)\geq \inf_\tau \ee_x\left\{\int_0^{\tau\wedge \sigma} f(X_s)ds + \ind{\tau < \sigma} w(X_{\tau})+ \ind{\sigma\leq \tau}z(X_\sigma) \right\}
\end{equation}
for any stopping time $\sigma$.
The inverse inequality is rather obvious so that we finally obtain \eqref{aeq3}.
\end{proof}

\bibliographystyle{siamplain}

\begin{thebibliography}{10}

\bibitem{Alv}
{\sc L.H.R. Alvarez}, {\em Stochastic forest stand value and optimal timber harvesting}, SIAM J. Control Optim. 42 (6) (2004), pp. 1972--1993.

\bibitem{BasSte2018}
{\sc A.~Basu and {\L}.~Stettner}, {\em Zero-sum {M}arkov games with impulse
  controls}, SIAM J. Control Optim. 58 (2020), pp.~580--604.

\bibitem{Bay2013}
{\sc E. Bayraktar, T. Emmerling and J.L. Menaldi}, {\em On the impulse control of jump duffusions}, SIAM J. Control Optim. 51 (2013), pp.~2612--2637.

\bibitem{Belak}
{\sc Ch. Belak, S. Christensen}, {\em Utility maximisation in a factor model with constant and proportional
transaction costs}, Finance Stoch. 23 (1) (2019),  29--96.


\bibitem{Ben2012}
{\sc A.  ~Bensoussan}, {\em Dynamic programming and inventory control. Studies in Probability, Optimization and Statistics, 3}, IOS Press, Amsterdam, 2011.

\bibitem{BenLio1984}
{\sc A.~Bensoussan and J.-L. Lions}, {\em Impulse Control And Quasi-Variational
  Inequalities}, Gauthier-Villars, Montrouge, 1984.

\bibitem{Ch2017}
{\sc S. Christensen, A. Irle and A.  Ludwig}, {\em Optimal portfolio selection under vanishing fixed transaction costs}, Adv. in Appl. Probab. 49 (2017), pp.~1116--1143.

\bibitem{Ch202}
{\sc S. Christensen and T. Sohr}, {\em A solution technique for L\'evy driven long term average impulse control problems}, Stochastic Process Appl. 130 (2020), pp. 7303--7337.

\bibitem{Dev2020}
{\sc A. Devraj, I. Kontoyiannis and S. Meyn}, {\em Geometric Ergodicity in a Weighted Sobolev Space},
Annals of Prob. 48 (2020), pp. 380--403.


\bibitem{DPS}
{\sc T. Duncan, B. Pasik-Duncan and L. Stettner}, {Growth optimal portfolio selection under proportional transaction costs with obligatory diversification}, App. Math. Optim. 63 (2010), pp. 107--132.

\bibitem{Dyn1965}
{\sc E.B.~Dynkin}, {\em Markov processes}, Springer, 1965.

\bibitem{GS}
{\sc D. Gatarek and L. Stettner}, {\em On the compactness method in general ergodic impulsive control of Markov processes}, Stoch. Stoch. Rep. 31 (1990), pp. 15--26.

\bibitem{Korn}
{\sc R. Korn}, {\em Some applications of impulse control in mathematical finance}, Math. Methods Oper. Res. 50 (3) (1999), 493--518.

\bibitem{JZ}
{\sc A. Jack and M. Zervos}, {\em Impulse control of one-dimensional Ito diffusions with an expected and a pathwise ergodic criterion}, Appl. Math. Optim. 54 (2006), pp. 71--93

\bibitem{LionsPert1986}
{\sc P.L. Lions and B. Perthame}, {\em Quasi-variational inequalities and ergodic impulse control},
SIAM J. Control Optim., 24 (1986), pp. 604--615.

\bibitem{Mac2011}
{\sc L.C. MacLean, E.O. Thorp and W.T.  Ziemba}, {\em The Kelly Capital Growth Investment Criterion}, World Scientific, River Edge, NJ, 2011.

\bibitem{MenRob2018}
{\sc J. Menaldi and M. Robin}, {\em On some ergodic impulse control problems with constraints}, SIAM J. Control Optim., 56 (2018), pp. 2690--2711.

\bibitem{Mun} {\sc G. Mundaca and B. Øksendal}, {\em Optimal stochastic intervention control with application to the exchange
rate}, J. Math. Econom., 29 (2) (1998), pp. 225--243.

\bibitem{PalSte2010}
{\sc J.~Palczewski and {\L}.~Stettner}, {\em Finite horizon optimal stopping of
  time-discontinuous functionals with applications to impulse control with
  delay}, SIAM J. Control Optim., 48 (2010), pp.~4874--4909.

\bibitem{PalSte2011} {\sc J.~Palczewski, {\L}.~Stettner}, {\em Stopping of discontinuous functionals with the first exit time discontinuity}, Stochastic Process. Appl.121, no.10, (2011), pp.~2361--2392.

\bibitem{PalSte2017}
{\sc J.~Palczewski and {\L}.~Stettner}, {\em Impulse control maximizing average
  cost per unit time: A nonuniformly ergodic case}, SIAM J. Control
  Optim., 55 (2017), pp.~936--960.

\bibitem{Pert1988} {\sc B. Perthame}, {\em Vanishing impulse control in the quasivariational inequality for ergodic impulse control}, Asymptot. Anal., 1 (1988), pp. 13--21.

\bibitem{Rob1978}
{\sc M.~Robin}, {\em Contr\^{o}le impulsionnel des processus de {M}arkov},
  th\`{e}se d'\'{e}tat, Universit{\'e} Paris Dauphine, 1978.
\newblock {Available at \url{https://tel.archives-ouvertes.fr/tel-00735779}}.

\bibitem{Rob1981}
{\sc M.~Robin}, {\em On some impulse control problems with long run average
  cost}, SIAM J. Control Optim., 19 (1981), pp.~333--358.

\bibitem{Robin1983}
{\sc M. Robin}, {\em Long-term average control problems for continuous time Markov Processes: A survey},
Acta Appl. Math. 1 (1983), pp. 281--299.

\bibitem{Stettner1983a} {\sc {\L}. ~Stettner}, {\em On Impulsive Control with Long Run Average Cost
Criterion}, Studia Math. 76 (1983), pp.~279--298.

\bibitem{Ste1986} {\sc {\L}. Stettner}, {\em On ergodic impulse control problems}, Stochastics, 18 (1986), pp. 49--72.


\bibitem{Ste2011} {\sc {\L}. ~Stettner}, {\em Penalty method for finite horizon stopping
problems}, SIAM J. Control Optim., 49 (2011), pp.~1078--1999.

\bibitem{Wil}
{\sc Y. Willassen}, {\em The stochastic rotation problem: A generalization of Faustmann’s formula to stochastic forest growth}, J. Econom. Dynam. Control, 22 (1998), pp. 573--596.

\end{thebibliography}

\end{document}